\documentclass[a4paper]{article}
\pdfoutput=1
\usepackage{graphicx, verbatim, enumerate, amsmath, amsthm, amssymb, amstext, color, amsfonts, enumitem}
\usepackage[arrow, matrix, curve]{xy}
\usepackage[applemac]{inputenc}

\theoremstyle{plain}
\newtheorem{theo}{Theorem}

\newtheorem{lem}[theo]{Lemma}

\newtheorem{cor}[theo]{Corollary}

\newtheorem{prop}[theo]{Proposition}

\theoremstyle{definition}

\theoremstyle{definition}

\newcommand{\N}{\ensuremath{\mathbb{N}}}
\newcommand{\F}{\ensuremath{\mathbb{F}}}
\newcommand{\Z}{\ensuremath{\mathbb{Z}}}

\newcommand{\sm}{\ensuremath{\smallsetminus}}

\newcommand{\sub}{\subseteq}

\newcommand{\COMMENT}[1]{}

\newcommand{\BF}{\ensuremath{\mathcal B}}
\newcommand{\CFtop}{\ensuremath{\mathcal C}_{\rm top}}
\newcommand{\CF}{\ensuremath{\mathcal C}}

\newcommand{\EF}{\ensuremath{\mathcal E}}
\newcommand{\FF}{\ensuremath{\mathcal F}}

\newcommand{\PF}{\ensuremath{\mathcal P}}

\newcommand{\RF}{\ensuremath{\mathcal R}}

\newenvironment{txteq}
  {
    \begin{equation}
    \begin{minipage}[c]{0.85\textwidth} % set width to 0.9 x textwidth
    \em                                % switch on emph
  }
  {\end{minipage}\end{equation}\ignorespacesafterend}
\newenvironment{txteq*}
  {
    \begin{equation*}
    \begin{minipage}[c]{0.85\textwidth} % set width to 0.9 x textwidth
    \em                                % switch on emph
  }
  {\end{minipage}\end{equation*}\ignorespacesafterend}

\DeclareMathOperator{\fin}{fin}
\DeclareMathOperator{\sk}{skew}
\DeclareMathOperator{\alg}{alg}

\DeclareMathOperator{\coloneqq}{\mathrel{\mathop:}=}

\advance\textheight by -1cm% to make page fit on my monitor at 2000
\title{Orthogonality and minimality\\ in the homology of locally finite graphs}
\author{Reinhard Diestel \and Julian  Pott}
\begin{document}
\maketitle
\begin{abstract}
Given a finite set $E$, a subset $D\sub E$ (viewed as a function $E\to \F_2$) is orthogonal to a given subspace $\FF$ of the $\F_2$-vector space of functions $E\to \F_2$ as soon as $D$ is orthogonal to every $\sub$-minimal element of~$\FF$.
This fails in general when $E$ is infinite.

However, we prove the above statement for the six subspaces $\FF$ of the edge space of any $3$-connected locally finite graph that are relevant to its homology: the topological, algebraic, and finite cycle and cut spaces. This solves a problem of~\cite{RDsBanffSurvey}.
\end{abstract}

\section{Introduction}%
   \COMMENT{}
Let $G$ be a $2$-connected locally finite graph, and let $\EF=\EF(G)$ be its edge space over $\F_2$.
We think of the elements of $\EF$ as sets of edges, possibly infinite.
Two sets of edges are \emph{orthogonal} if their intersection has (finite and) even cardinality.
A set $D\in\EF$ is \emph{orthogonal} to a subspace $\FF\sub \EF$ if it is orthogonal to every $F\in\FF$.
See \cite{DiestelBook10noEE, RDsBanffSurvey} for any definitions not given below.

The topological \emph{cycle space} $\CFtop(G)$ of~$G$ is the subspace of $\EF(G)$ generated (via thin sums, possibly infinite) by the \emph{circuits} of~$G$, the edge sets of the topological circles in the Freudenthal compactification $|G|$ of~$G$.
This space $\CFtop(G)$ contains precisely the elements of~$\EF$ that are orthogonal to $\BF_{\fin}(G)$, the finite-cut space of~$G$~\cite{DiestelBook10noEE}. The \emph{algebraic cycle space} $\CF_{\alg}(G)$ of~$G$ is the subspace of $\EF$ consisting of the edge sets inducing even degrees at all the vertices. It contains precisely the elements of~$\EF$ that are orthogonal to the \emph{skew cut space} $\BF_{\sk}(G)$~\cite{CR08}, the subspace of $\EF$ consisting of all the cuts of $G$ with one side finite. The \emph{finite-cycle space} $\CF_{\fin}(G)$ is the subspace of $\EF$ generated (via finite sums) by the finite circuits of~$G$. 
This space $\CF_{\fin}(G)$ contains precisely the elements of~$\EF$ that are orthogonal to $\BF(G)$, the cut space of~$G$~\cite{DiestelBook10noEE, RDsBanffSurvey}. Thus,
\[
\CFtop  =  \BF_{\fin}^{\perp},\quad
\CF_{\alg} = \BF_{\sk}^\perp,\quad
\CF_{\fin}  =  \BF^{\perp}\text{.}
\]
\vskip-6pt \noindent
 Conversely,
\[
\CFtop^{\perp}= \BF_{\fin},\quad
\CF_{\alg}^\perp = \BF_{\sk},\quad
\CF_{\fin}^{\perp}=\BF\text{.}
\]\vskip6pt\noindent
Thus, for any of the six spaces $\FF$ just mentioned, we have $\FF^{\perp\perp}=\FF$. 

\goodbreak

Proofs of most of the above six identities were first published by Casteels and Richter~\cite{CR08}, in a more general setting.%
   \COMMENT{}
   Any remaining proofs can be found in~\cite{RDsBanffSurvey}, except for the inclusion~$\CF_{\alg}^\perp\supseteq \BF_{\sk}$, which is easy.%
   \COMMENT{}

The six subspaces of~$\EF$ mentioned above are the the ones most relevant to the homology of locally finite infinite graphs. See~\cite{RDsBanffSurvey}, Diestel and Spr\"ussel~\cite{Hom1}, and Georgakopoulos~\cite{ltop,lhom}.
Our aim in this note is to facilitate orthogonality proofs for these spaces by showing that, whenever $\FF$ is one of them, a set $D$ of edges is orthogonal to $\FF$ as soon as it is orthogonal to the minimal nonzero elements of~$\FF$.

This is easy when $\FF$ is $\CF_{\fin}$ or $\BF_{\fin}$ or $\BF_{\sk}$:

\begin{prop}
Let $\FF$ be a subspace of $\EF$ all whose elements are finite sets of edges.
Then $\FF$ is generated (via finite sums) by its $\subseteq$-minimal nonzero elements.
\end{prop}
\begin{proof}
For a contradiction suppose that some $F\in\FF$ is not a finite sum of finitely many minimal nonzero elements of $\FF$.
Choose $F$ with $|F|$ minimal.
As $F$ is not minimal itself, by assumption, it properly contains a minimal nonzero element $F'$ of $\FF$.
As $F$ is finite, $F + F'=F\sm F'\in\FF$ has fewer elements than $F$, so there is a finite family $(M_i)_{i\le n}$ of minimal nonzero elements of $\FF$ with $\sum_{i\le n} M_i=F+F'$.
This contradicts our assumption, as $F'+\sum_{i\le n} M_i=F$.
\end{proof}

\begin{cor}\label{cor:finite}
If $\FF\in\{\CF_{\fin}, \BF_{\fin}, \BF_{\sk}\}$, a set $D$ of edges is orthogonal to $\FF$ as soon as $D$ is orthogonal to all the minimal nonzero elements of~$\FF$.\qed
\end{cor}

When $\FF\in\{\CFtop, \CF_{\alg},\BF\}$, the statement of Corollary~\ref{cor:finite} is generally false for graphs that are not $3$-connected. Here are some examples.

For $\FF=\BF$, let $G$ be the graph obtained from the $\N\times\Z$ grid by doubling every edge between two vertices of degree $3$ and subdividing all the new edges.
The set $D$ of the edges that lie in a $K^3$ of~$G$ is orthogonal to every bond $F$ of~$G$: their intersection $D\cap F$ is finite and even.
But $D$ is not orthogonal to every element of $\FF=\BF$, since it meets some cuts that are not bonds infinitely.

For $\FF=\CFtop$, let $B$ be an infinite bond of the infinite ladder $H$, and let $G$ be the graph obtained from $H$ by subdividing every edge in~$B$.
Then the set $D$ of edges that are incident with subdivision vertices has a finite and even intersection with every topological circuit $C$, finite or infinite, but it is not orthogonal to every element of~$\CFtop$, since it meets some of them infinitely.

For $\FF=\CF_{\alg}$ we can re-use the example just given for~$\CFtop$, since for 1-ended graphs like the ladder the two spaces coincide.

However, if $G$ is $3$-connected, an edge set is orthogonal to every element of $\CFtop, \CF_{\alg}$ or $\BF$ as soon as it is orthogonal to every minimal nonzero element:

\begin{theo}\label{maindual}
Let $G=(V,E)$ be a locally finite $3$-connected graph, and $F,D\sub E$.
\begin{enumerate}[label=\emph{(\roman*)}]
\item $F\in \CFtop^\perp$ as soon as $F$ is orthogonal to all the minimal nonzero elements of $\CFtop$, the topological circuits of~$G$.
\item $F\in \CF_{\alg}^\perp$ as soon as $F$ is orthogonal to all the minimal nonzero elements of $\CF_{\alg}$, the finite circuits and the edge sets of double rays in~$G$.
\item $D\in\BF^\perp $ as soon as $D$ is orthogonal to all the minimal nonzero elements of $\BF$, the bonds of~$G$.
\end{enumerate}
\end{theo}

Although Theorem~\ref{maindual} fails if we replace the assumption of $3$-connectedness with $2$-connectedness, it turns out that we need a little less than $3$-connectedness.
Recall that an end $\omega$ of~$G$ has (combinatorial) \emph{vertex-degree $k$} if $k$ is the maximum number of vertex-disjoint rays in~$\omega$.
Halin \cite{halin74} showed that every end in a $k$-connected locally finite graph has vertex-degree at least $k$.
Let us call an end $\omega$ of~$G$ \emph{$k$-padded} if for every ray $R\in\omega$ there is a neighbourhood $U$ of~$\omega$ such that for every vertex $u\in U$ there is a \emph{$k$-fan} from $u$ to $R$ in~$G$, a subdivided $k$-star with centre $u$ and leaves on~$R$.\footnote{For example, if $G$ is the union of complete graphs $K_1,K_2,\dots$ with $|K_i|=i$, each meeting the next in exactly one vertex (and these are all distinct), then the unique end of~$G$ is $k$-padded for every $k\in\N$.}
If every end of~$G$ is $k$-padded, we say that $G$ is \emph{$k$-padded at infinity}.
Note that $k$-connected graphs are $k$-padded at infinity.
Our proof of Theorem~\ref{maindual}(i) and~(ii) will use only that every end has vertex-degree at least $3$ and that $G$ is $2$-connected.
Similarly, and in a sense dually, our proof of Theorem~\ref{maindual}(iii) uses only that every end has vertex-degree at least $2$ and $G$ is $3$-connected at infinity.

\begin{theo}\label{mainbulky}
Let $G=(V,E)$ be a locally finite $2$-connected graph.
\begin{enumerate}[label=\emph{(\roman*)}]
\item If every end of~$G$ has vertex-degree at least $3$, then $F\in \CFtop^\perp$ as soon as $F$ is orthogonal to all the minimal nonzero elements of $\CFtop$, the topological circuits of~$G$.
\item If every end of~$G$ has vertex-degree at least $3$, then $F\in \CF_{\alg}^\perp$ as soon as $F$ is orthogonal to all the minimal nonzero elements of $\CF_{\alg}$, the finite circuits and the edge sets of double rays in~$G$.
\item If $G$ is $3$-padded at infinity, then $D\in\BF^\perp $ as soon as $D$ is orthogonal to all the minimal nonzero elements of $\BF$, the bonds of~$G$.
\end{enumerate}
\end{theo}

In general, our notation follows~\cite{DiestelBook10noEE}. 
In particular, given an end $\omega$ in a graph~$G$ and a finite set $S\sub V(G)$ of vertices, we write $C(S,\omega)$ for the unique component of $G-S$ that contains a ray $R\in\omega$.
The \emph{vertex-degree} of $\omega$ is the maximum number of vertex-disjoint rays in~$\omega$.
The mathematical background required for this paper is covered in~\cite{RDsBanffSurvey,Hom1}.
For earlier results on the cycle and cut space see Bruhn and Stein~\cite{Degree,BruhnSteinEndDuality}.

\section{Finding disjoint paths and fans}

Menger's theorem that the smallest cardinality of an $A$--$B$ separator in a finite graph is equal to the largest cardinality of a set of disjoint $A$--$B$ paths trivially extends to infinite graphs.
Thus in a locally finite $k$-connected graph, there are $k$ internally disjoint paths between any two vertices.
In Lemmas~\ref{lem:disjoint_fans} and~\ref{cor:connected_witness} we show that, for two such vertices that are close to an end $\omega$, these connecting paths need not use vertices too far away from~$\omega$.

In a graph $G$ with vertex sets $X,Y\sub V(G)$ and vertices $x,y\in V(G)$,
a \emph{$k$-fan} from $X$ (or $x$) to~$Y$ is a subdivided $k$-star whose center lies in~$X$ (or is $x$) and whose leaves lie in~$Y$.
A \emph{$k$-linkage} from $x$ to $y$ is a union of~$k$ internally disjoint $x$--$y$ paths.
We may refer to a sequence $(v_i)_{i\in\N}$ simply by $(v_i)$, and use $\bigcup(v_i):=\bigcup_{i\in\N}\{v_i\}$ for brevity.

\begin{lem}\label{lem:disjoint_fans}
Let $G$ be a locally finite graph with an end~$\omega$, and let $(v_i)_{i\in\N}$ and $(w_i)_{i\in\N}$ be two sequences of vertices converging to~$\omega$.
Let $k$ be a positive integer.
\begin{enumerate}[label=\emph{(\roman*)}]
\item If for infinitely many $n\in\N$ there is a $k$-fan from $v_n$ to $\bigcup(w_i)$, then there are infinitely many disjoint such $k$-fans.
\item If for infinitely many $n\in\N$ there is a $k$-linkage from $v_n$ to $w_n$, then there are infinitely many disjoint such $k$-linkages.
\end{enumerate}
\end{lem}

\begin{proof}
For a contradiction, suppose $k\in\N$ is minimal such that there is a locally finite graph $G=(V,E)$ with sequences $(v_i)_{i\in\N}$ and $(w_i)_{i\in\N}$ in which either (i) or (ii) fails.
Then $k>1$, since for every finite set $S\sub V(G)$ the unique component $C(S,\omega)$ of $G-S$ that contains rays from $\omega$ is connected and contains all but finitely many vertices from $\bigcup(v_i)$ and $\bigcup(w_i)$.

For a proof of (i) it suffices to show that for every finite set $S\sub V(G)$ there is an integer $n\in \N$ and a $k$-fan from $v_n$ to~$\bigcup (w_i)$ avoiding~$S$.
Suppose there is a finite set $S\sub V(G)$ that meets all $k$-fans from $\bigcup(v_i)$ to~$\bigcup(w_i)$.
By the minimality of~$k$, there are infinitely many disjoint $(k-1)$-fans from $\bigcup(v_i)$ to~$\bigcup(w_i)$ in $C \coloneqq C(S,\omega)$.
Thus, there is a subsequence $(v'_i)_{i\in\N}$ of $(v_i)_{i\in\N}$ in~$C$ and pairwise disjoint $(k-1)$-fans $F_i\sub C$ from $v'_i$ to~$\bigcup(w_i)$ for all $i\in\N$.
For every $i\in\N$ there is by Menger's theorem a $(k-1)$-separator $S_i$ separating $v'_i$ from $\bigcup(w_i)$ in~$C$, as by assumption there is no $k$-fan from $v'_i$ to~$\bigcup(w_i)$ in~$C$.
Let $C_i$ be the component of $G-(S\cup S_i)$ containing $v_i'$.

Since $F_i$ is a subdivided $|S_i|$-star, $S_i\sub V(F_i)$.
Hence for all $i\ne j$, our assumption of $F_i\cap F_j=\emptyset$ implies that $F_i\cap S_j=\emptyset$, and hence that $F_i\cap C_j=\emptyset$.
But then also $C_i\cap C_j=\emptyset$, since any vertex in $C_i\cap C_j$ coud be joined to $v_j'$ by a path $P$ in $C_j$ and to $v_i'$ by a path $Q$ in $C_i$, giving rise to a $v_j'$--$\bigcup (w_i)$ path in $P\cup Q\cup F_i$ avoiding $S_j$, a contradiction.%
  \COMMENT{}

As $S\cup S_i$ separates $v'_i$ from $\bigcup(w_i)$ in~$G$ and there is, by assumption, a $k$-fan from $v'_i$ to~$\bigcup(w_i)$ in~$G$, there are at least $k$ distinct neighbours of~$C_i$ in $S\cup S_i$.
Since $|S_i|=k-1$, one of these lies in~$S$. 
This holds for all $i\in\N$.
As $C_i\cap C_j=\emptyset$ for distinct $i$ and $j$, this contradicts our assumption that~$G$ is locally finite and $S$ is finite.
This completes the proof of (i).

For (ii) it suffices to show that for every finite set $S\sub V(G)$ there is an integer $n\in\N$ such that there is a $k$-linkage form $v_n$ to $w_n$ avoiding $S$.
Suppose there is a finite set~$S\sub V(G)$ that meets all $k$-linkages from $v_i$ to $w_i$ for all $i\in\N$.
By the minimality of $k$ there is an infinite family $(L_i)_{i\in I}$ of disjoint $(k-1)$-linkages $L_i$ in $C\coloneqq C(S,\omega)$ from $v_i$ to $w_i$.
As earlier, there are pairwise disjoint $(k-1)$-sets $S_i\sub V(L_i)$ separating $v_i$ from $w_i$ in~$C$, for all $i\in I$.
Let $C_i,D_i$ be the components of $C-S_i$ containing $v_i$ and $w_i$, respectively.
For no $i\in I$ can both $C_i$ and $D_i$ have $\omega$ in their closure, as they are separated by the finite set $S\cup S_i$.
Thus for every $i\in I$ one of $C_i$ or $D_i$ contains at most finitely many vertices from $\bigcup_{i\in I} L_i$.
By symmetry, and replacing $I$ with an infinite subset of itself if necessary, we may assume the following: 
\begin{txteq}\label{finitely_many}
The components $C_i$ with $i\in I$ each contain only finitely many\\ vertices from $\bigcup_{i\in I} L_i$.
\end{txteq}

If infinitely many of the components $C_i$ are pairwise disjoint, then $S$ has infinitely many neighbours as earlier, a contradiction.
By Ramsey's theorem, we may thus assume that 
\begin{equation}\label{not_empty}
C_i\cap C_j\neq\emptyset \text{ \em for all }i,j\in I.
\end{equation}

Note that if $C_i$ meets $L_j$ for some $j\ne i$, then $C_i\supseteq L_j$, since $L_j$ is disjoint from $L_i\supseteq S_i$.
By~(\ref{finitely_many}), this happens for only finitely many $j>i$.
We can therefore choose an infinite subset of~$I$ such that $C_i\cap L_j=\emptyset$ for all $i<j$ in $I$.
In particular, $(C_i\cup S_i)\cap S_j =\emptyset$ for $i<j$.
By~(\ref{not_empty}), this implies that 
\begin{equation}\label{nested}
C_i\cup S_i\sub C_j \text{ \em for all } i<j.
\end{equation}

By assumption, there exists for each $i\in I$ some $v_i$--$w_i$ linkage of $k$ independent paths in~$G$, one of which avoids $S_i$ and therefore meets $S$.
Let $P_i$ denote its final segment from its last vertex in $S$ to $w_i$.
As $w_i\in C\sm (C_i\cup S_i)$ and $P_i$ avoids both $S_i$ and $S$ (after its starting vertex in~$S$), we also have 
\begin{equation}\label{empty}
P_i\cap C_i=\emptyset.
\end{equation}

On the other hand, $L_i$ contains $v_i\in C_i\subseteq C_{i+1}$ and avoids $S_{i+1}$, so $w_i\in L_i\sub C_{i+1}$.
Hence $P_i$ meets $S_j$ for every $j\ge i+1$ such that $P_i\not\sub S\cup C_j$.
Since the $L_j\supseteq S_j$ are disjoint for different $j$, this happens for only finitely many $j>i$.
Deleting those $j$ from $I$, and repeating that argument for increasing $i$ in turn, we may thus assume that $P_i\sub S\cup C_{i+1}$ for all $i\in I$.
By (3) and (4) we deduce that $P_i\sm S$ are now disjoint for different values of $i\in I$.
Hence $S$ contains a vertex of infinite degree, a contradiction.
\end{proof}

Recall that $G$ is \emph{$k$-padded} at an end $\omega$ if for every ray $R\in\omega$ there is a neighbourhood $U$ such that for all vertices $u\in U$ there is a $k$-fan from $u$ to~$R$ in~$G$.
Our next lemma shows that, if we are willing to make $U$ smaller, we can find the fans locally around $\omega$:

\begin{lem}\label{cor:connected_witness}\label{lem:connected_witness}
Let $G$ be a locally finite graph with a $k$-padded end $\omega$.
For every ray $R\in\omega$ and every finite set $S\sub V(G)$ there is a neighbourhood $U\subseteq C(S,\omega)$ of $\omega$ such that from every vertex $u\in U$ there is a $k$-fan in $C(S,\omega)$ to $R$.
\end{lem}
\begin{proof}
Suppose that, for some $R\in\omega$ and finite $S\sub V(G)$, every neighbourhood $U\sub C(S,\omega)$ of~$\omega$ contains a vertex $u$ such that $C(S,\omega)$ contains no $k$-fan from $u$ to $R$.
Then there is a sequence $u_1,u_2,\dots$ of such vertices converging to $\omega$.
As $\omega$ is $k$-padded there are $k$-fans from infinitely many $u_i$ to $R$ in~$G$.
By Lemma~\ref{lem:disjoint_fans}(i) we may assume that these fans are disjoint.
By the choice of $u_1,u_2,\dots$, all these disjoint fans meet the finite set $S$, a contradiction.
\end{proof}

\section{The proof of Theorems~\ref{maindual} and \ref{mainbulky}}\label{proofofmaindual}

As pointed out in the introduction, Theorem~\ref{mainbulky} implies Theorem~\ref{maindual}.
It thus suffices to prove Theorem~\ref{mainbulky}, of which we prove (i) first.
Consider a set $F\neq\emptyset$ of edges that meets every circuit of~$G$ evenly.
We have to show that $F\in \CFtop^{\perp}$, i.e., that $F$ is a finite cut.
(Recall that $\CFtop^\perp$ is known to equal $\BF_{\fin}$, the finite-cut space~\cite{RDsBanffSurvey}.)
As $F$ meets every finite cycle evenly it is a cut, with bipartition $(A,B)$ say.
Suppose $F$ is infinite.
Let $\RF$ be a set of three disjoint rays that belong to an end $\omega$ in the closure of~$F$.
Every $R$--$R'$ path $P$ for two distinct $R,R'\in\RF$ lies on the unique topological circle $C(R,R',P)$ that is contained in $R\cup R'\cup P\cup\{\omega\}$.
As every circuit meets $F$ finitely, we deduce that no ray in $\RF$ meets $F$ again and again.
Replacing the rays in~$\RF$ with tails of themselves as necessary, we may thus assume that $F$ contains no edge from any of the rays in $\RF$.
Suppose $F$ separates $\RF$, with the vertices of $R\in\RF$ in~$A$ and the vertices of $R',R''\in\RF$ in $B$ say.
Then there are infinitely many disjoint $R$--$(R'\cup R'')$ paths each meeting $F$ at least once.
Infinitely many of these disjoint paths avoid one of the rays in~$B$, say $R''$.
The union of these paths together with $R$ and $R'$ contains a ray $W\in\omega$ that meets $F$ infinitely often.
For every $R''$--$W$ path $P$, the circle $C(W,R'',P)$ meets $F$ in infinitely many edges, a contradiction.
Thus we may assume that $F$ does not separate $\RF$, and that $G[A]$ contains $\bigcup\RF$.

As $\omega$ lies in the closure of $F$, there is a sequence $(v_i)_{i\in\N}$ of vertices in~$B$ converging to~$\omega$.
As~$G$ is $2$-connected there is a $2$-fan from each $v_i$ to $\bigcup \RF$ in~$G$.
By Lemma~\ref{lem:disjoint_fans} there are infinitely many disjoint $2$-fans from $\bigcup(v_i)$ to~$\bigcup\RF$.
We may assume that every such fan has at most two vertices in $\bigcup\RF$.
Then infinitely many of these fans avoid some fixed ray in $\RF$, say $R$. 
The two other rays plus the infinitely many $2$-fans meeting only these together contain a ray $W\in\omega$ that meets $F$ infinitely often and is disjoint from $R$.
Then for every $R$--$W$ path $P$ we get a contradiction, as $C(R,W,P)$ is a circle meeting $F$ in infinitely many edges.

For a proof of~(ii), note first that the minimal elements of $\CF_{\alg}$ are indeed the finite circuits and the edge sets of double rays in~$G$. Indeed, these are clearly in~$\CF_{\alg}$ and minimal. Conversely, given any element of~$\CF_{\alg}$, a set $D$ of edges inducing even degrees at all the vertices, we can greedily find for any given edge $e\in D$ a finite circuit or double ray with all its edges in~$D$ that contains~$e$. We may thus decompose $D$ inductively into disjoint finite circuits and edge sets of double rays, since deleting finitely many such sets from~$D$ clearly produces another element of~$\CF_{\alg}$, and including in each circuit or double ray chosen the smallest undeleted edge in some fixed enumeration of~$D$ ensures that the entire set $D$ is decomposed. If $D$ is minimal in~$\CF_{\alg}$, it must therefore itself be a finite circuit or the edge set of a double ray.

Consider a set $F$ of edges that fails to meet some set $D\in\CF_{\alg}$ evenly; we have to show that $F$ also fails to meet some finite circuit or double ray evenly. If $|F\cap D|$ is odd, then this follows from our decomposition of $D$ into disjoint finite circuits and edges sets of double rays. We thus assume that $F\cap D$ is infinite. Since $|G|$ is compact, we can find a sequence $e_1,e_2,\dots$ of edges in $F\cap D$ that converges to some end~$\omega$. Let $R_1,R_2,R_3$ be disjoint rays in~$\omega$, which exist by our assumption that $\omega$ has vertex-degree at least~3. Subdividing each edge~$e_i$ by a new vertex~$v_i$, and using that $G$ is 2-connected, we can find for every~$i$ a 2-fan from $v_i$ to $W = V(R_1\cup R_2\cup R_3)$ that has only its last vertices and possibly~$v_i$ in~$W$. By Lemma~\ref{lem:disjoint_fans}, with $w_1,w_2,\dots$ an enumeration of~$W$, some infinitely many of these fans are disjoint. Renaming the rays~$R_i$ and replacing $e_1,e_2,\dots$ with a subsequence as necessary, we may assume that either all these fans have both endvertices on~$R_1$, or that they all have one endvertex on~$R_1$ and the other on~$R_2$. In both cases all these fans avoid~$R_3$, so we can find a ray~$R$ in the union of $R_1$, $R_2$ and these fans (suppressing the subdividing vertices~$v_i$ again) that contains infinitely many~$e_i$ and avoids~$R_3$. Linking $R$ to a tail of~$R_3$ we thus obtain a double ray in $G$ that contains infinitely many~$e_i$, as desired.

To prove (iii), let $D\sub E$ be a set of edges that meets every bond evenly.
We have to show that $D\in\BF^\perp$, i.e., that $D$ has an (only finite and) even number of edges also in every cut that is not a bond.

As~$D$ meets every finite bond evenly, and hence every finite cut, it lies in $\BF_{\fin}^\perp=\CFtop$.
We claim that 
\begin{equation}\tag{$\star$}\label{disjoint_union}
  \begin{minipage}[c]{0.85\textwidth} 
	  \emph{$D$ is a disjoint union of finite circuits.}
  \end{minipage}
\end{equation}

To prove~(\ref{disjoint_union}), let us show first that every edge $e\in D$ lies in some finite circuit $C\sub D$.
If not, the endvertices $u,v$ of $e$ lie in different components of $(V, D\sm\{e\})$, and we can partition $V$ into two sets $A,B$ so that $e$ is the only $A$--$B$ edge in $D$.
The cut of~$G$ of all its $A$--$B$ edges is a disjoint union of bonds~\cite{DiestelBook10noEE}, one of which meets $D$ in precisely $e$.
This contradicts our assumption that $D$ meets every bond of~$G$ evenly.

For our proof of~(\ref{disjoint_union}), we start by enumerating $D$, say as $D=:\{e_1,e_2,\dots\}=:D_0$.
Let $C_0\sub D_0$ be a finite circuit containing $e_0$, let $D_1:= D_0\sm C_0$, and notice that $D_1$, like $D_0$, meets every bond of~$G$ evenly (because $C_0$ does).
As before, $D_1$ contains a finite circuit $C_1$ containing the edge $e_i$ with $i=\min \{j\mid e_j\in D_1\}$.
Continuing in this way we find the desired decomposition $D=C_1\cup C_2\cup\dots$ of $D$ into finite circuits.
This completes the proof of~(\ref{disjoint_union}).

As every finite circuit lies in $\BF^\perp$, it suffices by~(\ref{disjoint_union}) to show that $D$ is finite.
Suppose $D$ is infinite, and let $\omega$ be an end of~$G$ in its closure.
Let us say that two rays $R$ and $R'$ \emph{hug} $D$ if every neighbourhood $U$ of~$\omega$ contains a finite circuit $C\sub D$ that is neither separated from $R$ by~$R'$ nor from $R'$ by~$R$ in~$U$.
%some neighbourhood $U'\sub U$ of $\omega$ containing $C$.
We shall construct two rays $R$ and $R'$ that hug~$D$, inductively as follows.

Let $S_0=\emptyset$, and let $R_0,R'_0$ be disjoint rays in~$\omega$. (These exist as~$G$ is $2$-connected~\cite{halin74}.)
For step $j\ge 1$, assume that let $S_i, R_i$, and $R'_i$ have been defined for all $i<j$
so that $R_i$ and $R_i'$ each meet $S_i$ in precisely some initial segement (and otherwise lie in $C(S_i,\omega)$) and $S_i$ contains the $i$th vertex in some fixed enumeration of $V$.
If the $j$th vertex in this enumeration lies in $C(S_{j-1},\omega)$, add to $S_{j-1}$ this vertex and, if it lies on $R_{j-1}$ or $R'_{j-1}$, the initial  segement of that ray up to it.
Keep calling the enlarged set $S_{j-1}$.
For the following choice of $S$ we apply Lemma~\ref{cor:connected_witness} to $S_{j-1}$ and each of $R_{j-1}$ and $R'_{j-1}$.
Let $S\supseteq S_{j-1}$ be a finite set such that from every vertex $v$ in $C(S,\omega)$ there are $3$-fans in $C(S_{j-1},\omega)$ both to~$R_{j-1}$ and to $R'_{j-1}$.
By~(\ref{disjoint_union}) and the choice of $\omega$, there is a finite circuit $C_j\sub D$ in $C(S,\omega)$.
Then $C_j$ can not be separated from $R_{j-1}$ or $R'_{j-1}$ in $C(S_{j-1},\omega)$ by fewer than three vertices, and thus there are three disjoint paths from $C_j$ to $R_{j-1}\cup R'_{j-1}$ in $C(S_{j-1},\omega)$.

There are now two possible cases.
The first is that in~$C(S_{j-1},\omega)$ the circuit $C_j$ is neither separated from $R_{j-1}$ by~$R'_{j-1}$ nor from $R'_{j-1}$ by~$R_{j-1}$.
This case is the preferable case.
In the second case one ray separates $C_j$ from the other.
In this case we will reroute the two rays to obtain new rays as in the first case.
We shall then `freeze' a finite set containing initial parts of these rays, as well as paths from each ray to $C_j$.
This finite fixed set will not be changed in any later step of the construction of $R$ and $R'$.
In detail, this process is as follows.

If $C(S_{j-1},\omega)$ contains both a $C_j$--$R_{j-1}$ path $P$ avoiding $R'_{j-1}$  and a $C_j$--$R'_{j-1}$ path $P'$ avoiding $R_{j-1}$, let $Q$ and $Q'$ be the initial segments of $R_{j-1}$ and $R'_{j-1}$ up to $P$ and $P'$, respectively.
Then let $R_j=R_{j-1}$ and $R'_j=R'_{j-1}$ and 
\[S_j=S_{j-1}\cup V(P)\cup V(P')\cup V(Q)\cup V(Q').\]
This choice of $S_j$ ensures that the rays $R,R'$ constructed form the $R_i$ and $R_i'$ in the limit will not separate each other from $C_j$, because they will satisfy $R\cap S_j= R_j\cap S_j$ and $R'\cap S_j= R'_j\cap S_j$. 

If the ray $R_{j-1}$ separates $C_j$ from $R'_{j-1}$, let $\PF_j$ be a set of three disjoint $C_j$--$R'_{j-1}$ paths avoiding $S_{j-1}$.
All these paths meet $R_{j-1}$.
Let $P_1\in\PF_j$ be the path which $R_{j-1}$ meets first, and $P_3\in\PF_j$ the one it meets last.  
Then $R_{j-1} \cup C_j \cup P_1 \cup P_3$ contains a ray $R_j$ with initial segment $R_{j-1}\cap S_{j-1}$ that meets $C_j$ but is disjoint from the remaining path $P_2\in\PF$ and from $R_{j-1}'$.
Let $R'_{j}=R'_{j-1}$, and let $S_j$ contain $S_{j-1}$ and all vertices of $\bigcup\PF_j$, and the initial segments of $R_{j-1}$ and $R'_{j-1}$ up to their last vertex in $\bigcup \PF$.
Note that $R_j$ meets $C_j$, and that $P_2$ is a $C_j$--$R'_j$ path avoiding $R_j$.

If the ray $R'_{j-1}$ separates $C_j$ from $R_{j-1}$, reverse their roles in the previous  part of the construction.

The edges that lie eventually in $R_i$ or $R'_i$ as $i\to\infty$ form two rays $R$ and $R'$ that clearly hug $D$.

Let us show that there are two disjoint combs, with spines $R$ and $R'$ respectively, and infinitely many disjoint finite circuits in $D$ such that each of the combs has a tooth in each of these circuits.
We build these combs inductively, starting with the rays $R$ and $R'$ and adding teeth one by one.

Let $T_0=R$ and $T'_0=R'$ and $S_0=\emptyset$.
Given $j\ge 1$, assume that $T_i$, $T'_i$ and $S_i$ have been defined for all $i< j$.
By Lemma~\ref{cor:connected_witness} there is a finite set $S\supseteq S_{j-1}$ such that every vertex of $C(S,\omega)$ sends a $3$-fan to $R\cup R'$ in $C(S_{j-1},\omega)$.
As $R$ and $R'$ hug $D$ there is a finite cycle $C$ in $C(S,\omega)$ with edges in~$D$, and which neither of the rays $R$ or $R'$ separates from the other.
By the choice of~$S$, no one vertex of $C(S_{j-1},\omega)$ separates $C$ from $R\cup R'$ in $C(S_{j-1},\omega)$.
Hence by Menger's theorem there are disjoint $(R\cup R')$--$C$ paths $P$ and $Q$ in $C(S_{j-1},\omega)$.
If $P$ starts on $R$ and $Q$ starts on $R'$ (say), let $P':=Q$.
Assume now that $P$ and $Q$ start on the same ray $R$ or $R'$, say on $R$.
Let $Q'$ be a path from $R'$ to $C\cup P\cup Q$ in $C(S_{j-1},\omega)$ that avoids $R$.
As $Q'$ meets at most one of the paths $P$ and $Q$, we may assume it does not meet $P$.
Then $Q'\cup (Q\sm R)$ contains an $R'$--$C$ path $P'$ disjoint from $P$ and $R$.
In either case, let $T_j= T_{j-1}\cup P$, let $T_j'=T'_{j-1}\cup P'$, and let $S_j$ consist of $S_{j-1}$, the vertices in $C\cup P\cup P'$, and the vertices on $R$ and $R'$ up to their last vertex in $C\cup P\cup P'$. 

The unions $T=\bigcup_{i\in\N}T_i$ and $T'=\bigcup_{i\in\N}T'_i$ are disjoint combs that have teeth in infinitely many common disjoint finite cycles whose edges lie in~$D$.
Let $A$ be the vertex set of the component of~$G-T$ containing $T'$, and let $B:= V\sm A$.
Since $T$ is connected, $E(A,B)$ is a bond, and its intersection with $D$ is infinite as every finite cycle that contains a tooth from both these combs meets $E(A,B)$ at least twice.
This contradiction implies that $D$ is finite, as desired.
\qed

\bibliographystyle{plain}
\bibliography{collective}

\begin{thebibliography}{1}

\bibitem{Degree}
H.~Bruhn and M.~Stein.
\newblock On end degrees and infinite circuits in locally finite graphs.
\newblock {\em Combinatorica}, 27:269--291, 2007.

\bibitem{BruhnSteinEndDuality}
Henning Bruhn and Maya Stein.
\newblock Duality of ends.
\newblock {\em Comb., Probab. Comput.}, 19:47--60, 2010.

\bibitem{CR08}
K.~Casteels and B.~Richter.
\newblock {The Bond and Cycle Spaces of an Infinite Graph}.
\newblock {\em J.~Graph Theory}, 59(2):162--176, 2008.

\bibitem{DiestelBook10noEE}
R.~Diestel.
\newblock {\em {Graph Theory}}.
\newblock Springer, 4th edition, 2010.

\bibitem{RDsBanffSurvey}
R.~Diestel.
\newblock Locally finite graphs with ends: a topological approach.
\newblock {\em Discrete Math.}, 310--312:~2750--2765 (310); 1423--1447 (311);
  21--29 (312), 2010--11.
\newblock arXiv:0912.4213.

\bibitem{Hom1}
R.~Diestel and P.~Spr\"ussel.
\newblock The homology of locally finite graphs with ends.
\newblock {\em Combinatorica}, 30:681--714, 2010.

\bibitem{ltop}
A.~Georgakopoulos.
\newblock Graph topologies induced by edge lengths.
\newblock In Diestel, Hahn, and Mohar, editors, {\em Infinite graphs:
  introductions, connections, surveys}, volume Discrete Mathematics 311, pages
  1523--1542, 2011.

\bibitem{lhom}
A.~Georgakopoulos.
\newblock Cycle decompositions: from graphs to continua.
\newblock {\em Advances in Mathematics}, 229:935--967, 2012.

\bibitem{halin74}
R.~Halin.
\newblock A note on {M}enger's theorem for infinite locally finite graphs.
\newblock {\em Abh.\ Math.\ Sem.\ Univ.\ Hamburg}, 40:111--114, 1974.

\end{thebibliography}
\end{document}